\documentclass[a4paper,10pt]{amsart}

\usepackage[english]{babel}
\usepackage[utf8]{inputenc}

\usepackage{amssymb}
\usepackage{amsmath}
\usepackage{amsthm}
\usepackage{bbm}

\usepackage{enumitem}
\usepackage{tikz}
\usepackage{marginnote}

\newcommand{\bbC}{\mathbb{C}}

\newcommand{\bbQ}{\mathbb{Q}}
\newcommand{\bbR}{\mathbb{R}}

\newcommand{\bbT}{\mathbb{T}}

\newcommand{\bbZ}{\mathbb{Z}}

\newcommand{\calC}{\mathcal{C}}

\newcommand{\calL}{\mathcal{L}}

\newcommand{\calS}{\mathcal{S}}

\DeclareMathOperator{\id}{id}

\newcommand{\norm}[1]{\left\lVert #1 \right\rVert}
\newcommand{\modulus}[1]{\left\lvert #1 \right\rvert}

\newcommand{\spec}{\sigma}

\newcommand{\spr}{\operatorname{r}}

\theoremstyle{definition}
\newtheorem{definition}{Definition}[section]

\newtheorem{example}[definition]{Example}

\theoremstyle{plain}
\newtheorem{proposition}[definition]{Proposition}
\newtheorem{lemma}[definition]{Lemma}
\newtheorem{theorem}[definition]{Theorem}
\newtheorem{corollary}[definition]{Corollary}

\numberwithin{equation}{section}

\begin{document}

\title[Automatic continuity of semigroups]{Automatic time continuity of positive matrix and operator semigroups}
\author{Jochen Gl\"uck}
\address{Jochen Gl\"uck, University of Wuppertal, School of Mathematics and Natural Sciences, Gaußstr.\ 20, 42119 Wuppertal, Germany}
\email{glueck@uni-wuppertal.de}
\date{\today}
\begin{abstract}
	We consider a matrix semigroup $T: [0,\infty) \to \bbR^{d \times d}$ 
	without assuming any measurability properties and show that, 
	if $T$ is bounded close to $0$ and $T(t) \ge 0$ entrywise for all $t$, 
	then $T$ is continuous. 
	This complements classical results for the scalar-valued case. 
	We also prove an analogous result if $T$ takes values in the positive operators 
	over a sequence space.
\end{abstract}

\maketitle

\section{Introduction} 
\label{section:introduction}

\subsection*{Motivation}

It is a classical theme of ideas that a function $T: [0,\infty) \to \bbR$ that satisfies the functional equation 
\begin{align}
	\label{eq:fe-scalar}
	T(0) = 1, 
	\qquad 
	T(t+s) = T(t)T(s) \quad \text{for all } s,t \in [0,\infty)
\end{align}
is, under mild conditions, given by $T(t) = e^{ta}$ for a number $a \in \bbR$ and all $t \in [0,\infty)$. 
For instance, this is true if $T$ is continuous at $0$. 
Alternatively, it suffices to assume that $T$ is measurable and is not constantly zero on $(0,\infty)$. 
We refer to \cite[Sections~I.1 and~VII.2]{EngelNagel2000} for a more thorough discussion of this theme of ideas. 
In a similar vein, the following result holds, where measurability is replaced with a boundedness assumption close to $0$.

\begin{proposition}[Continuity of scalar semigroups]
	\label{prop:scalar}
	Let $T: [0,\infty) \to \bbR$ satisfy the functional equation~\eqref{eq:fe-scalar} 
	and assume that $T(t) \not= 0$ for at least one $t \in (0,\infty)$. 
	If $\sup_{t \in [0,1]} \modulus{T(t)} < \infty$, then $T$ is continuous 
	(and hence $T(t) = e^{ta}$ for a number $a \in \bbR$ and all $t \in [0,\infty)$).
\end{proposition}

This observation motivates the theorems in the following sections, 
and we also use it in an argument below. 
In order to be more self-contained, we thus include a proof of the continuity of $T$ in Proposition~\ref{prop:scalar} 
at the end of the introduction.
Note that Proposition~\ref{prop:scalar} does not remain true if $T$ maps to $\bbC$ instead of $\bbR$, 
see the simple Example~\ref{exa:complex} below. 
By identifying each complex number with an orthogonal matrix in $\bbR^{2 \times 2}$ we can thus see 
that the proposition does not hold for matrix-valued functions $T$, either.

In this note we show that the proposition remains valid, though, 
if $T$ is matrix-valued and satisfies $T(t) \ge 0$ entrywise for each $t \in [0,\infty)$ 
(Theorem~\ref{thm:matrix-pos}).
More generally, a similar result remains true if each $T(t)$ is a positive linear operator on a sequence space 
(Theorem~\ref{thm:sequences-pos} and Corollary~\ref{cor:sequences-pos}).

\subsection*{Matrix and operator semigroups}

We use the following notation. 
For a Banach space $E$, let $\calL(E)$ denote the space of bounded linear operators on $E$. 
If $E = \bbR^n$, we identify $\calL(E)$ with the matrix space $\bbR^{n \times n}$ in the usual way. 
We call a mapping $T: [0,\infty) \to \calL(E)$ a \emph{semigroup} or, 
more strictly speaking, a \emph{one-parameter operator semigroup}, if 
\begin{align*}
	T(0) = \id, 
	\qquad 
	T(t+s) = T(t)T(s) \quad \text{for all } s,t \in [0,\infty)
	.
\end{align*}
From an algebraic point of view, such a map is a monoid homomorphism 
from the additive monoid $[0,\infty)$ to the multiplicative monoid $\calL(E)$, 
but in operator theory it is common to refer to it by the notion \emph{semigroup} that we will use here. 
If $T(t)x \to x$ as $t \to 0$ for every $x \in E$, 
then it follows that each orbit of $T$ is continuous everywhere on $[0,\infty)$ 
\cite[Proposition~I.5.3]{EngelNagel2000} 
and in such a case, $T$ is called a \emph{strongly continuous semigroup} or a \emph{$C_0$-semigroup}. 
Such semigroups play an important role in the study of linear autonomous evolution equations.

If $E$ is finite-dimensional, strong continuity of $T$ is, of course, equivalent to continuity with respect to the operator norm 
and hence, every $C_0$-semigroup $T$ on a finite-dimensional space $E$ is given by 
$T(t) = e^{tA}$ for an $A \in \calL(E)$ and all $t \in [0,\infty)$ 
\cite[Theorem~I.3.7]{EngelNagel2000}.

\subsection*{Related literature}

There are various results of the type that a suitable measurability assumption implies time continuity 
not only in the scalar-valued case, but also for operator-valued semigrops;
see for instance \cite[Section~10.2]{HillePhillip1957}, \cite[Theorem~II.1]{Chung1967} and \cite[Corollary~3.3]{Gerlach2024}.
Hence, the main point about our results in Sections~\ref{sec:matrix} and~\ref{sec:sequence} 
is that they do not a priori assume any measurability with respect to the time parameter.

Automatic continuity of semigroups is related 
to the question which bounded linear operators can be embedded into a semigroup; 
the latter question is for instance studied in \cite{EisnerRadl2022}. 
The famous Markov embedding problem from the theory of Markov chains 
(see e.g.\ \cite{BaakeSumner2024, CasanellasFernandezSanchezRocaLacostena2023, Kingman1962}) 
is a particular case thereof.

\subsection*{An example and a proof of Proposition~\ref{prop:scalar}}

It is well-known that there exists an additive function $\varphi: \bbR \to \bbR$ which is not continuous. 
By composing it with the exponential function along the imaginary line, 
one can get a counterexample to the assertion of Proposition~\ref{prop:scalar} for complex-valued $T$. 
As the exponential function is not bijective along the imaginary line, a bit of care is needed, though, 
to check that the composition is indeed discontinuous. 
For the sake of completeness we include the details of the argument in the following example.
We formulate the example for a function $T$ defined on $\bbR$ 
but it can, of course, be restricted to $[0,\infty)$.

\begin{example}
	\label{exa:complex}
	Let $\varphi: \bbR \to \bbR$ be a discontinuous additive function and 
	define
	\begin{align*}
		T: \bbR \to \bbC, 
		\qquad 
		t \mapsto e^{i \varphi(t)}
		.
	\end{align*} 
	Then $T$ satisfies the functional equation~\eqref{eq:fe-scalar}, but $T$ is not continuous at any point.
\end{example}

\begin{proof}
	Obviously, $T$ satisfies~\eqref{eq:fe-scalar}.
	Thus, it suffices to show that there exists one point at which $T$ is not continuous.
	So assume to the contrary that $T$ is continuous on $\bbR$. 
	Then there exists a number $a \in \bbR$ such that $T(t) = e^{iat}$ for all $t \in \bbR$. 
	Thus
	\begin{align}
		\label{eq:additive-function-in-grid}
		\varphi(t) - at \in 2\pi \bbZ
	\end{align}
	for each $t \in \bbR$. 
	Since $\varphi$ is additive, it follows by algebraic induction that $\varphi(qt) = q\varphi(t)$ 
	for each $t \in \bbR$ and each $q \in \bbQ$. 
	By substituting $qt$ for $t$ in~\eqref{eq:additive-function-in-grid} we thus obtain
	\begin{align*}
		q(\varphi(t) - at) \in 2\pi \bbZ
	\end{align*}
	for all $t \in \bbR$ and all $q \in \bbQ$. 
	If we fix $t$ and choose $q$ sufficiently small, we can conclude that $\varphi(t) - at = 0$. 
	So $\varphi(t) = at$ for each $t \in \bbR$, 
	which contradicts the choice of $\varphi$ as a discontinuous function.
\end{proof}

A simple proof of Proposition~\ref{prop:scalar} is as follows.

\begin{proof}[Proof of Proposition~\ref{prop:scalar}]
	It suffices to prove that $T$ is continuous at $0$.
	First note that the assumption $T(t) \not= 0$ for at least one $t \in (0,\infty)$ 
	together with the functional equation~\eqref{eq:fe-scalar} implies that $T(t) \not= 0$ for all $t \in (0,\infty)$. 
	Moreover, one has $T(t) = T(t/2)^2 > 0$ for all $t \in [0,\infty)$.
	
	Now set $M := \sup_{t \in [0,1]} T(t)$. 
	Since $M < \infty$ by assumption, it follows from the functional equation~\eqref{eq:fe-scalar} 
	that $\limsup_{t \downarrow 0} T(t) \le 1$. 
	Let us assume towards a contradiction that $\liminf_{t \downarrow 0} T(t) < 1$. 
	
	We claim that this implies $T(1) = 0$ (which thus gives a contradiction). 
	To see this, fix a number $\delta$ that is strictly between $\liminf_{t \downarrow 0} T(t)$ and $1$, 
	and let $\varepsilon > 0$ be arbitrary. 
	Choose an integer $n \ge 1$ such that $\delta^n \le \varepsilon$. 
	There exists a time $s \in [0, \frac{1}{n}]$ such that $T(s) \le \delta$. 
	Thus, $T(ns) \le \delta^n \le \varepsilon$. 
	Since $ns \le 1$, it follows that
	\begin{align*}
		T(1) = T(1-ns)T(ns) 
		\le 
		M \varepsilon.
	\end{align*}
	As $\varepsilon$ was arbitrary we conclude that $T(1) = 0$.
\end{proof}

\section{Automatic continuity of matrix semigroups}
\label{sec:matrix}

The following is our main result in finite dimensions. 
For a matrix $B \in \bbR^{n \times n}$ we write $B \ge 0$ if the entries of $B$ satisfy $B_{jk} \ge 0$ for all $j,k$.

\begin{theorem}[Continuity of positive matrix semigroups]
	\label{thm:matrix-pos}
	Let $T: [0,\infty) \to \bbR^{n \times n}$ be a semigroup with the following properties:
	\begin{enumerate}[label=\upshape(\arabic*)]
		\item 
		For at least one $t \in (0,\infty)$ the matrix $T(t)$ is invertible.
		
		\item 
		For each $t \in [0,\infty)$ one has $T(t) \ge 0$. 
		
		\item 
		The semigroup $T$ is bounded close to $0$, i.e.\ $M := \sup_{t \in [0,1]} \norm{T(t)} < \infty$.
	\end{enumerate}
	Then $T$ is continuous.
\end{theorem}

The following observation was made in \cite[Corollary to Proposition~7]{Kingman1962}: 
let $T: [0,\infty) \to \bbR^{n \times n}$ be a semigroup that takes values in the stochastic matrices 
and such that $T(1)$ is invertible; 
then there exists a continuous semigroup $\tilde T: [0,\infty) \to \bbR^{n \times n}$ 
that also takes values in the stochastic matrices and satisfies $\tilde T(1) = T(1)$. 
Our Theorem~\ref{thm:matrix-pos} shows that, in fact, $T$ is itself continuous in this situation.

For the proof of Theorem~\ref{thm:matrix-pos} we need the following lemma 
which is a consequence of the Perron--Frobenius theorem. 
For a matrix $B \in \bbR^{n \times n}$ we denote its spectrum by $\spec(B) \subseteq \bbC$.
Let $\bbT := \{z \in \bbC: \, \modulus{z} = 1\}$ denote the complex unit circle.

\begin{lemma} 
	\label{lem:pf}
	Let $T: [0,\infty) \to \bbR^{n \times n}$ be a semigroup such that $T(t) \ge 0$ for each $t$ 
	and $\sup_{t \in [0,\infty)} \norm{T(t)} < \infty$. 
	Then $\spec(T(t)) \cap \bbT \subseteq \{1\}$ for each $t \in [0,\infty)$.
\end{lemma}

\begin{proof}
	By assumption, each of the matrices $T(t)$ is power bounded, 
	i.e.\ it satisfies $\sup_{n \ge 1} \norm{T(t)^n} < \infty$; 
	hence, its spectral radius $\spr(T(t))$ is at most $1$.
	 
	If there exists a time $t \in (0,\infty)$ such that the spectral radius of $T(t)$ satisfies $\spr(T(t)) < 1$, 
	then the boundedness of $T$ implies that $T(s) \to 0$ as $s \to \infty$ and hence, 
	$T(s)^n \to 0$ as $n \to \infty$ for each $s \in (0,\infty)$. 
	Thus, $\spr(T(s)) < 1$ for each $s \in (0,\infty)$ and hence, the claim of the lemma is true. 
	
	So let us now assume that $r(T(t)) = 1$ for each $t \in (0,\infty)$. 
	Fix $t \in (0,\infty)$ and an eigenvalue $\lambda$ of $T(t)$ of modulus $\modulus{\lambda} = 1$.
	It follows from the Perron--Frobenius theorem that $\lambda$ is a root of unity 
	and that $\lambda^k$ is also an eigenvalue of $T(t)$ for all $k \in \bbZ$ 
	\cite[Theorem~I.2.7]{Schaefer1974}.
	Take the minimal integer $\ell \ge 1$ that satisfies $\lambda^\ell = 1$. 
	Since $T(t)$ has at most $n$ eigenvalues, it follows that $\ell \le n$. 
	So in particular, $\lambda^{n!} = 1$. 
	Note that $n$ is the dimension of the matrix space $\bbR^{n \times n}$, 
	so it does not depend on $t$ nor on the choice of $\lambda$.
	
	Since $T(t) = T(t/n!)^{n!}$, it follows from the spectral mapping theorem that there exists an eigenvalue $\mu$ of $T(t/n!)$ 
	such that $\lambda = \mu^{n!}$. 
	One has $\modulus{\mu} = 1$ and since the argument from the previous paragraph 
	applies to all times in $(0,\infty)$ and to all unimodular eigenvalues, 
	we conclude that $\lambda = \mu^{n!} = 1$, as claimed.
\end{proof}

As a second ingredient for the proof of Theorem~\ref{thm:matrix-pos} we use the following concept. 
A net $(x_j)$ in a set $X$ is called \emph{universal} if for each subset $S \subseteq X$ 
the net $(x_j)$ is either eventually contained in $S$ or eventually contained in $X \setminus S$. 
By using the existence of ultrafilters one can show that every net has a subnet which is a universal net. 
Employing this fact one can show that a topological space is compact if and only if every universal net in it converges.
Finally, for a function $f: X \to Y$ between two sets and a universal net $(x_j)$ in $X$, 
one readily checks that the net $\big(f(x_j)\big)$ is universal, too.

\begin{proof}[Proof of Theorem~\ref{thm:matrix-pos}]
	Since every net has a subnet that is a universal net, 
	it suffices to show that $T(t_j) \to \id$ for each universal net in $(0,\infty)$ that converges to $0$; 
	so let $(t_j)$ be such a universal net.
	
	For each $s \in [0,\infty)$ the net $(T(st_j))_{j}$ is universal and norm bounded by $M$; 
	by the compactness of bounded closed balls in finite-dimensional spaces, 
	the net thus converges to a matrix $S(s) \in \bbR^{n \times n}$ that satisfies $S(s) \ge 0$ and $\norm{S(s)} \le M$.
	One readily checks that $(S(s))_{s \in [0,\infty)}$ is a semigroup, 
	so it follows from Lemma~\ref{lem:pf} that $\spec(S(s)) \cap \bbT \subseteq \{1\}$ for each $s$.
	
	Now we show that each $S(s)$ has determinant $1$. 
	To this end, note that
	\begin{align*}
		[0,\infty) \to \bbR, 
		\qquad 
		t \mapsto \det(T(t)).
	\end{align*}
	is a semigroup and is thus continuous by Proposition~\ref{prop:scalar}.
	Hence, $\det(T(t)) \to 1$ for $t \to 0$ and thus, $\det(S(s)) = \lim_j \det(T(t_js)) = 1$ for each $s \in [0,\infty)$. 
	
	Let $s \in [0,\infty)$. 
	Since $S(s)$ is power bounded and has determinant $1$, 
	all eigenvalues of $S(s)$ are located in the unit circle $\bbT$ and therefore, $\spec(S(s)) = \{1\}$.
	Again since $S(s)$ is power-bounded, the eigenvalue $1$ of $S(s)$ is also semi-simple, so $S(s) = \id$. 
	In particular, $S(1) = \id$, so $T(t_j) \to \id$.
\end{proof}

One might wonder why we need $(t_j)$ in the proof to be a universal net instead of a sequence. 
For a sequence $(t_j)$ we would need to consider a subsequence of $(T(st_j))_{j}$ to obtain convergence 
-- but the subsequence would then depend on $s$ and thus, 
one could not expect $(S(s))_{s \in [0,\infty)}$ to be a semigroup.

\section{Automatic continuity on sequence spaces}
\label{sec:sequence}

The following result is a version of Theorem~\ref{thm:matrix-pos} on infinite-dimensional sequence spaces.
The most general version of the result can be stated on so-called \emph{atomic Banach lattices}, 
which are sometimes also called \emph{discrete Banach lattices}. 
Yet, readers not interested in this concept can simply ignore this notion 
and think of sequence spaces such as $c$ or $c_0$ or $\ell^p$ for $p \in [1,\infty]$, 
which are -- when endowed with the entrywise partial order -- 
the most important examples of atomic Banach lattices. 
The definition of an atomic Banach lattice can for instance be found 
in \cite[Section~2.5, p.\,113]{MeyerNieberg1991} or \cite[Section~0.2, p.\,18]{Wnuk1999}. 
We refrain from discussing, and even from recalling, the precise definition here 
since we will not explicitly work with properties of those spaces 
-- they only occur in the following theorem and its proof because we use \cite[Corollary~5.6(1)]{GlueckHaase2019},  
which tells us that groups of positive operators on such spaces are trivial under appropriate assumptions.

For a Banach lattice $E$, an operator $B \in \calL(E)$ is called \emph{positive} if $f \ge 0$ implies $Bf \ge 0$ 
for each $f \in E$. 

\begin{theorem}[Automatic continuity on sequence spaces]
	\label{thm:sequences-pos}
	Let $E$ be one of the sequence spaces $c_0$, $c$ or $\ell^p$ for $p \in [1,\infty]$ or, more generally, 
	an atomic Banach lattice.
	Let $T: [0,\infty) \to \calL(E)$ be a positive semigroup with the following properties: 
	\begin{enumerate}[label=\upshape(\arabic*)]
		\item 
		For each non-zero $x \in E$ there exists at least one $t \in (0,\infty)$ such that $T(t)x \not= 0$.
		
		\item 
		For each $t \in [0,\infty)$ one has $T(t) \ge 0$.
		
		\item\label{thm:sequences-pos:itm:compact} 
		For each $x \in E$ the local orbit $\{T(t)x : \, t \in [0,1]\}$ is relatively compact.
	\end{enumerate}
	Then $T$ is strongly continuous, i.e.\ a $C_0$-semigroup.
\end{theorem}

\begin{proof}
	We start similarly as in the proof of Theorem~\ref{thm:matrix-pos}.
	It suffices to show that $\big(T(t_j)\big)$ converges strongly to $\id$ for each universal net in $(0,\infty)$ that converges to $0$; 
	so let $(t_j)$ be such a universal net.
	
	It follows from the assumption that the set $\{T(t): \, t \in [0,1]\}$ is relatively compact 
	with respect to the strong operator topology 
	(see e.g.\ \cite[Corollary~A.5]{EngelNagel2000}); 
	let $\calS$ denote the closure of this set in $\calL(E)$ with respect to the strong operator topology.
	For each $s \in [0,\infty)$ the universal net $\big(T(s t_j)\big)$ is eventually contained in the compact set $\calS$
	and thus converges strongly to an operator $S(s) \in \calS$. 
	The mapping $S: [0,\infty) \to \calL(E)$ is a semigroup 
	since operator multiplication is jointly strongly continuous on bounded subsets of $\calL(E)$.
	
	Next we show that, for each non-zero $x \in X$, the closure of the orbit $\{S(s): \, s \in [0,\infty)\}$ does not contain $0$. 
	Indeed, assume the contrary for some $0 \not= x \in X$. 
	By assumption we can find a time $t \in (0,\infty)$ such that $T(t)x \not= 0$. 
	By the semigroup law, we can even choose $t$ as small as we wish to, say $t \in (0,1]$.
	Let $\varepsilon > 0$ and set $M := \sup \{\norm{T(t)}: \, t \in [0,1]\}$. 
	Then $M < \infty$ by the uniform boundedness principle. 
	As assumed at the beginning of the paragraph, 
	there exists an $s \in [0,\infty)$ such that $\norm{S(s)x} \le \varepsilon$ 
	and hence, we can find an index $j$ such that $st_j \le t$ and $\norm{T(st_j)x} \le 2 \varepsilon$. 
	Therefore,
	\begin{align*}
		\norm{T(t)x} 
		\le 
		\norm{T(t-st_j)} \norm{T(st_j)x} 
		\le 
		2M\varepsilon
		,
	\end{align*}
	where we used that $0 \le t -st_j \le t \le 1$ and the definition of $M$. 
	Since $\varepsilon$ was arbitrary, it follows that $T(t)x = 0$, a contradiction.
	
	As each operator $S(s)$ is contained in the set $\calS$ which is compact with respect to the strong operator topology,
	one can apply the celebrated Jacobs--de Leeuw--Glicksberg decomposition 
	(see e.g.\ \cite[Section~2.4]{Krengel1985} or \cite[Section~16.3]{EisnerFarkasHaaseNagel2015})
	to the strong operator closure $\calC$ of the set $\{S(s): \, s \in [0,\infty)\}$. 
	The fact proved in the previous paragraph implies that the stable part of this decomposition is $0$, 
	so $\calC$ is in fact a strongly compact subgroup of the group of lattice isomorphisms on $E$.
	One readily checks that the group $\calC$ is \emph{divisible}, 
	i.e., for every $C \in \calC$ and every integer $n \ge 1$ there exists a $D \in \calC$ 
	such that $D^n = C$. 
	As proved in \cite[Corollary~5.6(1)]{GlueckHaase2019} the divisibility of $\calC$ 
	together with the fact that $E$ is atomic implies that $\calC$ consists of the identity operator only. 
	So in particular, $S(s) = \id$ for each $s \in [0,\infty)$ and hence, 
	$S(1) = \id$, which proves that $T(t_j) \to \id$ strongly.
\end{proof}

The sequence spaces $\ell^p$ for $p \in [1,\infty)$ and $c_0$ have, 
in contrast to the sequence spaces $\ell^\infty$ and $c$, the following nice property: 
for each $y \ge 0$ in such a space the so-called \emph{order interval} 
$[0,y] := \{x: \, 0 \le x \le y\}$ is compact. 
More generally, this is true in all atomic Banach lattices that have, in addition, 
\emph{order continuous norm}, which means that every decreasing net with infimum $0$ converges in norm to $0$; 
see e.g.\ \cite[Theorem~6.1(ii) and~(v)]{Wnuk1999} for a proof.

On such spaces, the compactness assumption from the previous theorem is implied by a local boundedness condition at $0$ 
together with the existence of a certain fixed point of the semigroup. 
By a \emph{fixed point} or \emph{fixed vector} of a semigroup $T: [0,\infty) \to \calL(E)$ 
on a Banach space $E$ we mean a vector $x \in E$ such that $T(t)x = x$ for all $t \in [0,\infty)$.
An element $x \ge 0$ of a Banach lattice $E$ is called a \emph{quasi-interior point} of the cone $E_+$ 
if the linear span of the order interval $[0,x]$ is norm dense in $E$. 
An element $x$ of $c_0$ or of $\ell^p$ for $p \in [1,\infty)$ can be checked 
to be a quasi-interior point of the cone if and only if all its entries $x_k$ are strictly positive.

\begin{corollary}[Automatic continuity on sequence spaces]
	\label{cor:sequences-pos}
	Let $E$ be one of the sequence spaces $c_0$ or $\ell^p$ for $p \in [1,\infty)$ or, more generally, 
	an atomic Banach lattice with order continuous norm.
	Let $T: [0,\infty) \to \calL(E)$ be a positive semigroup with the following properties: 
	\begin{enumerate}[label=\upshape(\arabic*)]
		\item 
		For each non-zero $x \in E$ there exists at least one $t \in (0,\infty)$ such that $T(t)x \not= 0$.
		
		\item 
		For each $t \in [0,\infty)$ one has $T(t) \ge 0$.
		
		\item\label{cor:sequences-pos:itm:bdd} 
		The semigroup $T$ is bounded closed to $0$, i.e.\ $M := \sup_{t \in [0,1]} \norm{T(t)} < \infty$.
		
		\item 
		There exists a fixed vector $h$ of $T$ that is a quasi-interior point of $E_+$.
	\end{enumerate}
	Then $T$ is strongly continuous, i.e.\ a $C_0$-semigroup.
\end{corollary}

\begin{proof}
	We show that the assumptions of Theorem~\ref{thm:sequences-pos} are satisfied; 
	it suffices to check assumption~\ref{thm:sequences-pos:itm:compact} of the theorem. 
	For every $x$ in the order interval $[0,h]$ one has $T(t)x \in [0,h]$ for all $t \in [0,1]$ 
	since $h$ is a fixed point and since each $T(t)$ is positive. 
	As $[0,h]$ is compact we conclude that $\{T(t)x: \, t \in [0,1]\}$ is relatively compact 
	for each $x \in [0,h]$ and hence, even for each $x$ in the linear span of $[0,h]$. 
	This span is dense in $E$ by assumption and thus, the boundedness condition 
	in~\ref{cor:sequences-pos:itm:bdd} implies that $\{T(t)x: \, t \in [0,1]\}$ is relatively compact 
	even for each $x \in E$ \cite[Corollary~A.5]{EngelNagel2000}. 
\end{proof}

\subsection*{Acknowledgements} 

I am indebted to Daniel Lenz who suggested to me several years ago 
that a result along the lines of Corollary~\ref{cor:sequences-pos} might be true.

\bibliographystyle{plain}
\bibliography{literature}

\end{document}